\documentclass[12pt]{amsart}
\usepackage{amssymb}
\title[An entropic proof of cutoff]{An entropic proof of cutoff on Ramanujan graphs}
\author[N. Ozawa]{Narutaka Ozawa}
\address{RIMS, Kyoto University, \mbox{606-8502} Japan}
\email{narutaka@kurims.kyoto-u.ac.jp}

\date{\today}
 
\newtheorem{thm}{Theorem}

\newtheorem{cor}[thm]{Corollary}
\newtheorem{lem}[thm]{Lemma}
\theoremstyle{remark}
\newtheorem{remark}[thm]{Remark}

\newcommand{\IB}{\mathbb B}
\newcommand{\IC}{\mathbb C}
\newcommand{\IE}[1]{\mathop{\mathbb E}\bigl[#1\bigr]}

\newcommand{\IN}{\mathbb N}

\newcommand{\IP}{\mathbb P}
\newcommand{\cA}{\mathcal A}
\newcommand{\cB}{\mathcal B}
\newcommand{\cN}{\mathcal N}

\newcommand{\cG}{\mathcal G}

\newcommand{\ve}{\varepsilon}

\newcommand{\mix}{\mathrm{mix}}
\newcommand{\tv}{\mathrm{TV}}

\DeclareMathOperator{\Prob}{Prob}
\DeclareMathOperator{\supp}{supp}
\newcommand{\ip}[1]{\mathopen{\langle}#1\mathclose{\rangle}}
\begin{document}
\begin{abstract}
It is recently proved by Lubetzky and Peres that the simple random 
walk on a Ramanujan graph exhibits a cutoff phenomenon, that is to say, 
the total variation distance of the random walk distribution from 
the uniform distribution drops abruptly from near $1$ to near $0$. 
There are already a few alternative proofs of this fact. 
In this note, we give yet another proof based on 
functional analysis and entropic consideration. 
\end{abstract}
\maketitle

\section{Introduction}
Throughout this note, we fix $d\geq 3$ and consider 
finite simple connected $d$-regular non-bipartite graphs $G$.
We identify $G$ with its vertex set and denote by $P \in \IB(\ell_2G)$ 
the normalized adjacency matrix on $G$, 
which is given by $P(x,y)=\frac{1}{d}$ if $x$ and $y$ are 
adjacent and else $0$. 
The hermitian matrix $P$ leaves the constant vectors invariant 
and the corresponding eigenvalue $1$ has multiplicity $1$. 
Thus the eigenvalues of $P$ are 
$1=\lambda_1>\lambda_2\geq\cdots\geq\lambda_{|G|}>-1$. 
We are interested in the \emph{reduced spectral radius}
\begin{equation}
\rho:=\| P|_{\ell_2G\cap(\IC 1)^\perp} \|_{\IB(\ell_2G\cap(\IC 1)^\perp)}
 = \max\{\lambda_2,-\lambda_{|G|}\}<1.
\end{equation}
We consider the simple random walk $(X^t)_{t\in\IN}$ on $G$ starting 
at some $X^0\in G$. 
The probability distribution $\mu^t:=P^t(\,\cdot\,,X^0)$ of $X^t$ 
converges to the uniform distribution $\pi:=\frac{1}{|G|}1_{G}$.
For $\alpha\in(0,1)$, the \emph{total variation mixing time} is 
defined to be 
\begin{equation}
T^{\mix}(\alpha):=\min\{ t \in \IN : \| \mu^t - \pi \|_{\tv} < \alpha\},
\end{equation}
where 
\begin{equation}
\| \nu-\eta \|_{\tv}=\max_{A\subset G}\{|\nu(A) - \eta(A)|\}
 =\frac{1}{2}\sum_{x\in G} |\nu(x)-\eta(x)|=\frac{1}{2}\|\nu-\eta\|_1
\end{equation}
is the total variation norm. 
The total variation mixing time $T^{\mix}(\alpha)$ is clearly monotone in $\alpha$.
Since the random walk $(X^t)_t$ backtracks with probability $1/d$, 
the expected distance from the origin $X^0$ to $X^t$ is at most 
$\frac{d-2}{d}t$. 
Thus $\mu^t$ is concentrated on the ball of radius 
$\frac{d-2}{d}t+O(\sqrt{t})$, which has cardinality 
at most $d(d-1)^{\frac{d-2}{d}t+O(\sqrt{t})}$.  
This observation leads to the \emph{entropic lower bound} 
of the total variation mixing time,
\begin{equation}\label{eq:elb}
T^{\mix}(\alpha) \geq \frac{\log|G|}{h_d}-o_{\alpha}(\log|G|), 
\end{equation}
where 
\begin{equation}
h_d:=\frac{(d-2)\log(d-1)}{d}
\end{equation}
is the asymptotic entropy of the simple random walk on the $d$-regular tree 
(see \cite{lp,bl} for a more precise estimate). 
On the other hand, 
since $\|\mu^t - \pi \|_1\le |G|^{1/2}\|\mu^t - \pi \|_2\le|G|^{1/2}\rho^t$, 
one has the \emph{spectral upper bound}
\begin{equation}\label{eq:sub}
T^{\mix}(\alpha) \le \frac{\log|G|+O_{\alpha}(1)}{-2\log\rho}\le \frac{\log|G|+O_{\alpha}(1)}{2(1-\rho)}.
\end{equation}
This upper bound also allows an entropic interpretation, see Lemma~\ref{entinc} below.

Now, let $\cG=\{ G \}$ be a family of graphs. 
Recall that $\cG$ is said to be a \emph{family of expanders} (see e.g., \cite{lubotzky}) if 
\begin{equation}
\sup\{ \rho_G : G\in\cG\} < 1.
\end{equation}
When needed, we add the subscript $G$ to the symbol $\rho$ etc 
to specify the graph $G$ under consideration.
Note that the \emph{Alon--Boppana bound} gives the lower estimate 
\begin{equation}
\liminf_{G\in\cG,\,|G|\to\infty} \rho_G \geq \frac{2\sqrt{d-1}}{d}=:\rho_d
\end{equation}
for any family $\cG$. 
Here the value $\rho_d$ is the spectral radius of the simple random 
walk on the $d$-regular tree. 
The family $\cG$ is said to be \emph{asymptotically Ramanujan} 
if it marks the Alon--Boppana bound,  
\begin{equation}
\lim_{G\in\cG,\,|G|\to\infty}\rho_G=\rho_d.
\end{equation}
An explicit family of Ramanujan graphs was first discovered by 
Lubotzky--Phillips--Sarnak (\cite{lps}) and it was shown by 
Friedman (\cite{friedman}) that the random graphs are asymptotically Ramanujan.
We have seen in the above that every expander family satisfies 
\begin{equation}\label{eq:expasym}
T^{\mix}_{G}(\alpha)\asymp_\alpha \log |G|
\end{equation}
as $|G|\to\infty$.
The family $\cG$ is said to \emph{exhibit cutoff} 
if the implied constants in the above asymptotic equality 
are independent of $\alpha$, or more precisely, if 
\begin{equation}
\lim_{G\in\cG,\,|G|\to\infty} \frac{T^{\mix}_{G}(\alpha)}{T^{\mix}_G(\alpha')}=1
\end{equation}
for every $\alpha,\alpha'\in(0,1)$ and every choice of the starting points $X_G^0 \in G$.
It is proved by Lubetzky and Peres (\cite{lp}) that asymptotically Ramanujan graphs exhibit cutoff. 
\begin{thm}\label{ramanujancutoffthm}
A family $\cG$ of asymptotically Ramanujan graphs exhibits cutoff. 
In fact, the total variation mixing time marks the entropic lower bound: 
For any $\alpha\in(0,1)$,  
\begin{equation}
\lim_{G\in\cG,\,|G|\to\infty}\frac{T^{\mix}_G(\alpha)}{\log |G|}=\frac{d}{(d-2)\log(d-1)}
=\frac{1}{h_d}.
\end{equation}
\end{thm}
See \cite{hermon,bl} for alternative proofs of this theorem.
In this note, we give yet another proof of it based on 
functional analysis and entropic consideration. 
Entropic considerations have been also instrumental 
in some of the recent works on cutoff phenomena, 
see \cite{bhlp,bhs,bcs,ck}, to just name a few.
\section{Entropy}

Let $\Prob(G)$ denote the set of probability measures on a  (finite or infinite) graph $G$.
Recall that the \emph{Shannon entropy} (or just \emph{entropy}) of $\nu\in\Prob(G)$ is the quantity 
\begin{equation}
H(\nu)=-\sum_x\nu(x)\log\nu(x).
\end{equation}
When a random walk $(X^t)_{t\in\IN}$ on $G$ is under consideration, 
we also write  
\begin{equation}
H(t) := H(\mu^t) =\IE{-\log\mu^t(X^t)}.
\end{equation}

Let's assume for the moment that the graph $G$ is transitive. 
Then the  quantity $H(t)$ is increasing and concave (Proposition 1.3 in \cite{kv}). 
The limit 
\begin{equation}
h:=\lim_t \frac{1}{t}H(t)=\lim_t H(t)-H(t-1)
\end{equation}
is called the \emph{asymptotic entropy} of the random walk $(X^t)_t$. 
Moreover, the Shannon--McMillan--Breiman type Theorem 
holds that $\frac{1}{t}|\mbox{$-\log\mu^t(X^t)$} - H(t)|$ converges to zero almost surely (Section IV in \cite{derriennic}, Theorem 2.1 in \cite{kv}). 
This result is vacuous when $G$ is finite.

In this note, we are interested in the case of a finite simple connected 
$d$-regular non-bipartite graph $G$ and the quantitative growth property of $H(t)$. 
Since $G$ is connected and non-bipartite, one has $\mu^t\to\pi$ 
and $H(t)\nearrow H(\pi)=\log|G|$.

We review a concavity property of the entropy functional. 
For $\nu\in\Prob(G)$, we put $P\nu:=\sum_x P(\,\cdot\,,x)\nu(x)\in\Prob(G)$. 
Note that $\mu^{t+1}=P\mu^t$. 
A caveat is in order here: the notation $\nu P$ instead of $P\nu$ may be more common 
among probability theory, but we stick to $P\nu$, because we view $\nu$ as a vector 
in $\ell_1$ and deal with $\nu^{1/2}\in\ell_2$. This notation should not 
cause any confusion, particularly because our $P$ is always symmetric. 

\begin{lem}\label{entinc}
For any $\nu\in\Prob(G)$, one has 
\begin{equation}
H(P\nu)-H(\nu) \geq \frac{1-\rho}{16}\|\nu-\pi\|_1^2.
\end{equation}
\end{lem}
\begin{proof}
By Hall's Matching Theorem, there are permutation matrices $P_i$ such that 
$P=\sum_{i=1}^d\lambda_i P_i$, where $\lambda_i=1/d$ (just to see 
that the proof works in more general setting). 
Thus $\bar{\nu}:=P\nu$ is the average of $\nu_i:=P_i\nu$.
Note that $H(\nu_i)=H(\nu)$ since $\nu_i$ has the same distribution up to a permutation as $\nu$. 
Put $f(t)=t\log t$ and observe that $f''(t)=1/t \geq 1/(a+b)$ for $t\in[a,b]$. 
For any $a_i\geq0$ and $\bar{a}=\sum_i \lambda_i a_i$, 
Taylor's theorem yields
\begin{equation}
\sum_i \lambda_i(f(a_i)-f(\bar{a}))
 \geq\frac{1}{2}\sum_i\lambda_i\frac{(a_i-\bar{a})^2}{a_i+\bar{a}}
 \geq\frac{1}{2}\sum_i\lambda_i(a_i^{1/2}-\bar{a}^{1/2})^2.
\end{equation}
It follows that 
\begin{equation}\label{eq:entconc}
H(\bar{\nu})-H(\nu)\geq  \frac{1}{2}\sum_i\lambda_i\|\nu_i^{1/2}-\bar{\nu}^{1/2}\|_2^2.
\end{equation}
Since $\nu_i^{1/2}=(P_i\nu)^{1/2}=P_i\nu^{1/2}$ and $\bar{\nu}^{1/2}$ are 
unit vectors in $\ell_2G$, 
\begin{equation}
\frac{1}{2}\sum_i\lambda_i\|\nu_i^{1/2}-\bar{\nu}^{1/2}\|_2^2
=1-\ip{P\nu^{1/2},\bar{\nu}^{1/2}}
\geq 1-\|P\nu^{1/2}\|_2.
\end{equation}
By Pythagorean theorem, 
for $\xi_0:=\nu^{1/2}-\ip{\nu^{1/2},\pi^{1/2}}\pi^{1/2}\in(\IC\pi^{1/2})^\perp\cap\ell_2G$,
\begin{equation}
1-\|P\nu^{1/2}\|_2 \geq \frac{1}{2}(\|\nu^{1/2}\|_2^2 - \|P\nu^{1/2}\|_2^2)
 \geq \frac{1}{2}(1 - \rho^2)\|\xi_0\|_2^2
\end{equation}
and 
\begin{equation}
\|\xi_0\|_2^2 = 1-\ip{\nu^{1/2},\pi^{1/2}}^2 
 \geq 1-\ip{\nu^{1/2},\pi^{1/2}} 
 =\frac{1}{2}\|\nu^{1/2}-\pi^{1/2}\|_2^2.
\end{equation}
Finally, one has 
\begin{equation}\label{eq:ell2ell1}
\|\nu^{1/2}-\pi^{1/2}\|_2
 \geq \frac{1}{2}\|\nu^{1/2}+\pi^{1/2}\|_2\|\nu^{1/2}-\pi^{1/2}\|_2
 \geq \frac{1}{2}\|\nu-\pi\|_1.
\end{equation}
By putting the above inequalities (\ref{eq:entconc})-(\ref{eq:ell2ell1}) together, we obtain the claim.
\end{proof}

Here is an entropic characterization of 
cutoff phenomena for expander graphs. 
It is widely believed that every expander family of transitive graphs should exhibit cutoff. 

\begin{thm}\label{cond}
Let $\cG=\{G\}$ be an expander family of finite simple connected 
$d$-regular non-bipartite graphs.
Then $\cG$ exhibits cutoff if and only if the following holds: 
For any $\ve>0$ and any $\delta>0$, there is $N\in\IN$ satisfying that 
\begin{equation}
H(T^{\mix}_G(1-\ve))>(1-\delta)\log|G|
\end{equation}
for every $G\in\cG$ with $|G|>N$.
\end{thm}
\begin{proof}
We prove the `if' part. 
Since $H(t+1)-H(t)\geq(1-\rho)\ve^2/16$ for all $t<T_G^{\mix}(\ve)$ 
by Lemma~\ref{entinc},  $H(T^{\mix}_G(1-\ve))>(1-\delta)\log|G|$ implies 
\begin{equation}
T_G^{\mix}(\ve)-T_G^{\mix}(1-\ve)
 <\delta \frac{16}{\ve^2} \frac{\log|G|}{1-\rho}.
\end{equation}
This proves the `if' part.
For the `only if' part, we observe that if $\|\nu-\pi\|_{\tv}<\delta/2$, 
then $A:=\{ x\in G : \nu(x)>\frac{2}{\delta|G|}\}$ 
satisfies $\nu(A) < \pi(A)+\delta/2 < \delta$, which implies 
\begin{equation}
H(\nu)\geq-\sum_{x\in G\setminus A}\nu(x)\log\nu(x)
 \geq (1-\delta)(\log|G| + \log(\delta/2)).
\end{equation}
This means that $H(T^{\mix}_G(\delta/2))\geq(1-\delta)\log|G| - O_\delta(1)$. 
In case the graphs $G\in\cG$ are transitive (but not necessarily expanders), 
one has by concavity that 
$H(t_2)-H(t_1)\le H(t_1)\frac{t_2-t_1}{t_1}$ for every $t_1<t_2$ 
and hence cutoff requires $H(T^{\mix}_G(1-\ve))>(1-\delta)\log|G|$ 
for $G$ large enough. 
In general, since $H(t+1)-H(t)\in[0,\log d]$ always holds, 
\begin{equation}
1-\delta-\frac{H(T^{\mix}_G(1-\ve))}{\log|G|} -\frac{O_\delta(1)}{\log |G|}
 \le \frac{T^{\mix}_G(\delta/2)-T^{\mix}_G(1-\ve)}{\log|G|}\log d\to 0,
\end{equation}
by (\ref{eq:expasym}), provided that $\cG$ exhibits cutoff. 
This proves the `only if' part.
\end{proof}

\begin{remark}\label{rem:referee}
The referee has pointed out that a similar result to Theorem~\ref{cond} 
has been used implicitly in the literature, e.g., in \cite{blps}. 
Here is a sketch of a more direct proof along this line, which is communicated by the referee. 
Let $\delta>0$ be given and 
$\nu\in\Prob(G)$ be such that $H(\nu)>(1-\delta)\log|G|$. 
Then for $A:=\{ x \in G : -\log\nu(x) \geq (1-\delta^{1/2})\log|G|  \}$ 
and $m:=\nu(A)$, one has 
\begin{equation}
H(\nu) < (1-m)(1-\delta^{1/2})\log|G| + m\log\frac{|G|}{m}, 
\end{equation}
which implies $1-m< \delta^{1/2}+O_\delta((\log|G|)^{-1})$. 
Thus one has $\|\nu-\pi - (\nu|_A-m\pi)\|_1\le2(1-m)$, 
$\|(\nu|_A-m\pi)\|_2\le2|G|^{-(1-\delta^{1/2})/2}$, 
and 
\begin{equation}
\|P^k\nu-\pi\|_{1}\le 2\rho_G^k|G|^{\delta^{1/2}/2}+2\delta^{1/2}+O_\delta((\log|G|)^{-1}) 
\le 4\delta^{1/2}
\end{equation}
at $k=\frac{\delta^{1/2}}{-2\log\rho_G}\log|G| + O_\delta(1)$ when $\log|G|$ is large enough.
\end{remark}

It is well-known to experts that a cutoff phenomenon is generally related to 
a measure concentration phenomenon for $- \log \mu^t(X^t)$. 

\begin{cor}\label{qsmb}
Let $\cG=\{G\}$ be an expander family of finite simple connected 
$d$-regular non-bipartite graphs. 
Assume that $\cG$ has the following quantitative 
Shannon--McMillan--Breiman type property: 
For any $\delta,\kappa>0$ there is $T_0\in\IN$ such that every $G\in\cG$ 
and $t\geq T_0$ satisfy 
\begin{equation}
\mu_G^t(\cN_{\delta\log|G|}(\{ x\in G : -\log\mu_G^t(x) < H(t)+\delta t\}))> 1-\kappa,
\end{equation}
where $\cN_r(A)$ denotes the $r$-neighborhood of $A\subset G$.
Then $\cG$ exhibits cutoff.
\end{cor}
\begin{proof}
Put $t:=T^{\mix}(1-\ve)$, $\kappa:=\ve/2$, 
and $A:=\{ x : -\log\mu^t(x) < H(t) + \delta t\}$. 
Then $|A|\le e^{H(t)+\delta t}$ and so $B:=\cN_{\delta\log|G|}(A)$ 
has cardinality $|B|\le d^{\delta\log|G|}e^{H(t)+\delta t}$.
For $t\geq T_0$, one has 
\begin{equation}
1-\ve>\|\mu^t-\pi\|_{\tv}
 \geq\mu^t(B)-\pi(B)>1-\ve/2-|G|^{-1+\delta\log d}e^{H(t)+\delta t}.
\end{equation}
This implies $e^{H(t)+\delta t}> |G|^{1-\delta\log d} \ve / 2$. 
Hence for $\beta:=t/\log |G|$ one has 
\begin{equation}
H(T_{\mix}(1-\ve))> (1-(\beta+\log d)\delta)\log |G| -O_\ve(1).
\end{equation} 
Since $\beta$ is bounded by (\ref{eq:sub}), 
this yields the cutoff phenomenon by Theorem~\ref{cond}.
\end{proof}
\section{Proof of Theorem~\ref{ramanujancutoffthm}}
In this section, we prove Theorem~\ref{ramanujancutoffthm}. 
In light of Theorem~\ref{cond} and the entropic lower bound (\ref{eq:elb}), 
it suffices to show $H(t)-H(t-1)\geq (1-\delta)h_d$ holds for $t\le T^{\mix}(1-\ve)$.
Upon change of $\alpha\in(0,1)$, we may replace $T^{\mix}(\alpha)$ by 
the \emph{Hellinger distance mixing time} 
\begin{equation}
T^{\mix,2}(\alpha):=\min\{t\in\IN : \|(\mu^t)^{1/2}-\pi^{1/2}\|_2^2<2\alpha\}.
\end{equation}
We state this well-known fact as a lemma for the readers' convenience. 
\begin{lem}
For every $\ve>0$, there is $\delta>0$ such that every $\nu\in\Prob(G)$ 
satisfies the following. 
\begin{itemize}
\item If $\|\nu-\pi\|_1<\delta$, then $\|\nu^{1/2}-\pi^{1/2}\|_2^2<\ve$.
\item If $\|\nu^{1/2}-\pi^{1/2}\|_2^2<\delta$, then $\|\nu-\pi\|_1<\ve$.
\item If $\|\nu-\pi\|_1<2-\ve$, then $\|\nu^{1/2}-\pi^{1/2}\|_2^2<2-\delta$.
\item If $\|\nu^{1/2}-\pi^{1/2}\|_2^2<2-\ve$, then $\|\nu-\pi\|_1<2-\delta$.
\end{itemize}
\end{lem}
\begin{proof}
The first three are easy consequences of 
\begin{equation}
\|\nu^{1/2}-\pi^{1/2}\|_2^2\le \|\nu-\pi\|_1\le\|\nu^{1/2}-\pi^{1/2}\|_2\|\nu^{1/2}+\pi^{1/2}\|_2.
\end{equation} 
For the last, put $A:=\{ x : \nu(x)>\frac{\ve^2}{16|G|}\}$.
Then $\|\nu^{1/2}-\pi^{1/2}\|_2^2<2-\ve$ implies
\begin{equation}
\ve/2< \ip{\nu^{1/2},\pi^{1/2}}\le(|A|/|G|)^{1/2}+\ve/4
\end{equation}
and hence $|A|\geq(\ve^2/16)|G|$. It follows that $\|\nu-\pi\|_1<2-\ve^4/128$.
\end{proof}
Here is another auxiliary lemma about concavity of the square root. 
\begin{lem}\label{ces}
Let $(\Omega,\cA,\IP)$ be a standard probability space and $\cB\subset\cA$ be a $\sigma$-subalgebra.
For any bounded random variable $f\geq 0$, the conditional expectation $g=\IE{f \,|\, \cB}$ satisfies 
\begin{equation}
\| g^{1/2} - f^{1/2} \|_{L^2(\Omega,\IP)}^2 \le 2\|f\|_{L^\infty(\Omega,\IP)}^{1/2}(\IE{g^{1/2}}-\IE{f^{1/2}}).
\end{equation}
\end{lem}
\begin{proof}
First, assume that $\cB$ is trivial and hence $g=\IE{f}$ is a constant.
Then, 
\begin{equation}
\| g^{1/2} - f^{1/2} \|_2^2
 =2g^{1/2}\int g^{1/2} - f^{1/2}\,d\IP
 \le 2\|f\|_\infty^{1/2}\int g^{1/2} - f^{1/2}\,d\IP
\end{equation}
and the assertion holds true.
The case where $\cB$ is purely atomic follows by integrating 
this estimate over atoms. 
The general case then follows by approximation (or the dis\-integration formula).
\end{proof}

\begin{proof}[Proof of Theorem~\ref{ramanujancutoffthm}]
Recall from the beginning of the section that we want to estimate
\begin{equation}\label{eq:entft}
H(t)-H(t-1)=\IE{-\log f_t},
\end{equation} 
where
\begin{equation}
f_t:=\frac{\mu^t(X^t)}{\mu^{t-1}(X^{t-1})}.
\end{equation}
The random variable $f_t$ satisfies $f_t\geq\frac{1}{d}$ and 
\begin{equation}
\IE{f_t}=\sum_{x,y}\frac{\mu^t(y)}{\mu^{t-1}(x)} \cdot P(y,x)\mu^{t-1}(x)
=\sum_{x,y\in G,\,x\in\supp\mu^{t-1}}\mu^t(y)P(y,x)\le1.
\end{equation}
Note that since $\sum_{x\in\supp\mu^{t-1}}P(y,x)=1$ for $y\in\supp\mu^{t-2}$, 
one has 
\begin{equation}
1- \IE{f_t} \le \sum_{y\notin\supp\mu^{t-2}}\mu^t(y)\le(\frac{d-1}{d})^{t-1}.
\end{equation}
To relate the random variable $f_t$ to the functional analysis on $\ell_2G$, 
we consider the unit vectors $\xi_t:=(\mu^t)^{1/2}$ in $\ell_2G$ and observe that 
\begin{equation}
\ip{ P\xi_{t-1},\xi_t }
 = \sum_{x,y\in G} P(y,x)\mu^{t-1}(x) \Bigl(\frac{\mu_t(y)}{\mu_{t-1}(x)}\Bigr)^{1/2} =\IE{f_t^{1/2}}.
\end{equation}
Hence, for any $t\le T^{\mix,2}(1-\ve)$, one has 
\begin{equation}
\IE{f_t^{1/2}} \le \|P\xi_{t-1}\|\le \ip{\xi_{t-1},\pi^{1/2}}+\rho<\rho+\ve.
\end{equation} 

We will compare the value $\IE{f_t^{1/2}}$ with that for the covering tree.
Let's consider the covering of $G$ by the $d$-regular tree $T_d$ 
and lift the random walk $(X^t)_t$ to the simple random walk $(\tilde{X}^t)_t$ on $T_d$. 
We write 
$\tilde{f}_t:=\tilde{\mu}^t(\tilde{X}^t)/\tilde{\mu}^{t-1}(\tilde{X}^{t-1})$ and 
observe that
\begin{equation}
\IE{\tilde{f}_t \,|\, X^{t-1},X^t } \approx f_t.
\end{equation}
Indeed, suppose $X^{t-1}=x$ and $X^t=y$. 
Then $x$ and $y$ are adjacent and 
$\IP(X^{t-1}=x,X^t=y)=\mu^{t-1}(x)/d$.
There is a bijection $q_{x,y}$ from the lifts $[x]$ 
of $x$ onto the lifts $[y]$ of $y$ such that 
$p$ and $q_{x,y}(p)$ are adjacent for all $p\in[x]$. 
Moreover, $\IP(\tilde{X}^{t-1}=p,\tilde{X}^t=q_{x,y}(p))=\tilde{\mu}^{t-1}(p)/d$.
Thus at the event $X^{t-1}=x$ and $X^t=y$, one has 
\begin{equation}
\IE{\tilde{f}_t \,|\, X^{t-1},X^t}
 =\frac{d}{\mu^{t-1}(x)}\sum_{p\in [x]\cap\supp\tilde{\mu}^{t-1}}\frac{\tilde{\mu}^t(q_{x,y}(p))}{\tilde{\mu}^{t-1}(p)}\frac{\tilde{\mu}^{t-1}(p)}{d}
\le\frac{\mu^t(y)}{\mu^{t-1}(x)}=f_t.
\end{equation}
Note that $\| f_t-\IE{\tilde{f}_t \,|\, X^{t-1},X^t}\|_1\le(\frac{d-1}{d})^{t-1}$ is virtually negligible.
By the concavity of the square root, one has 
\begin{equation}
\IE{f_t^{1/2}}\geq \IE{(\IE{\tilde{f}_t \,|\, X^{t-1},X^t})^{1/2}} \geq \IE{\tilde{f}_t^{1/2}}.
\end{equation}
It is not hard to show that the value of $\tilde{f}_t$ is asymptotically $d-1$ 
with probability $\frac{1}{d}$ (when the step $\tilde{X}^t$ is in the direction to the origin)  
and $\frac{1}{d-1}$ with probability $\frac{d-1}{d}$. 
Since $\tilde{f}_t\geq\frac{1}{d}$ and $\IE{\tilde{f}_t}\le1$ 
are unconditionally true, one has 
\begin{equation}
\lim_t\IE{-\log\tilde{f}_t}=-\frac{1}{d}\log(d-1)+\frac{d-1}{d}\log(d-1)=\frac{(d-2)\log(d-1)}{d}=h_d
\end{equation}
and
\begin{equation}
\lim_t \IE{\tilde{f}_t^{1/2}} = \frac{1}{d}\sqrt{d-1}+\frac{d-1}{d}\sqrt{\frac{1}{d-1}}
=\frac{2\sqrt{d-1}}{d} = \rho_d.
\end{equation}
The first equality merely confirms that the asymptotic entropy of 
the tree random walk $(\tilde{X}^t)_t$ is $h_d$. 
The second equality is the lucky(?) coincidence that we are going to exploit. 

Let $T_\ve\in\IN$ be such that $\IE{\tilde{f}_t^{1/2}}\geq\rho_d-\ve$ 
for all $t\geq T_\ve$.
By summarizing the above discussion, we obtain the following. 
For any $0<\ve<1$, if an asymptotically Ramanujan graph $G$ is large enough, 
then one has 
\begin{equation}\label{eq:sqrtapprox}
\rho_d+\ve \geq \IE{f_t^{1/2}}\geq \IE{\tilde{f}_t^{1/2}} \geq\rho_d-\ve
\end{equation}
for all $t\in [T_\ve, T_G^{\mix,2}(1-\ve)]$.
It is intuitively clear (and will be explained later) that this implies 
\begin{equation}\label{eq:entapprox}
\IE{-\log f_t} \geq (1-\delta(\ve))h_d
\end{equation}
for all $t\in [T_\ve, T_G^{\mix,2}(1-\ve)]$, where
$\delta\geq0$ is a continuous function with $\delta(0)=0$. 

To conclude the proof of Theorem~\ref{ramanujancutoffthm} 
via Theorem~\ref{cond}, 
let $\ve>0$ and $\delta>0$ be given. 
Take $\ve_0\in(0,\ve)$ small enough so that $\delta(\ve_0)<\delta$. 
Then by (\ref{eq:entft}), (\ref{eq:entapprox}), 
and the entropic lower bound (\ref{eq:elb}), one obtains 
\begin{equation}
H(T^{\mix,2}_G(1-\ve))\geq H(T^{\mix,2}_G(1-\ve_0))
\geq (1-\delta)\log|G|-o_{\ve_0}(\log|G|).
\end{equation}
This finishes the proof of Theorem~\ref{ramanujancutoffthm}.

For completeness, we explain why (\ref{eq:sqrtapprox}) implies (\ref{eq:entapprox}).
For this, we may assume that $\|f_t-\IE{\tilde{f}_t \,|\, X^{t-1},X^t } \|_1<\ve/2$ and 
$\| \tilde{f}_t - f\|_1<\ve/2$ for some random variable $f$ 
that takes value $d-1$ with probability $\frac{1}{d}$ and $\frac{1}{d-1}$ with probability $\frac{d-1}{d}$. 
Then $g:=\IE{f \,|\, X^{t-1},X^t}$ satisfies $\|g-f_t\|_1<\ve$ and 
\begin{equation}
\| g^{1/2} - f_t^{1/2} \|_1\le\| | g - f_t|^{1/2}\|_1\le \| g-f_t \|_1^{1/2}<\ve^{1/2}.
\end{equation}
By (\ref{eq:sqrtapprox}), 
this implies that $\IE{g^{1/2}}-\IE{f^{1/2}}\le2\ve+\ve^{1/2}<3\ve^{1/2}$ and, 
by Lemma~\ref{ces}, that 
\begin{equation}
\|g^{1/2}-f^{1/2}\|_2 \le (2\|f\|_\infty^{1/2}(\IE{g^{1/2}}-\IE{f^{1/2}}))^{1/2}\le
3(d-1)^{1/4}\ve^{1/4}.
\end{equation}
Thus, $\| g - f \|_1\le \|g^{1/2}-f^{1/2}\|_2\|g^{1/2}+f^{1/2}\|_2 = O_d(\ve^{1/4})$ 
and $\| f_t-f \|_1= O_d(\ve^{1/4})$. 
Since $f\geq\frac{1}{d-1}$, this implies $\IE{-\log f_t}\geq(1-\delta(\ve))\IE{-\log f}$ 
for some continuous function $\delta\geq0$ with $\delta(0)=0$.
\end{proof}

\begin{remark}
In the above proof, we saw $\IE{f_t^{1/2}}\approx\IE{\tilde{f}_t^{1/2}}$. 
This implies that $(f_t)_t$ are asymptotically independent family and, 
by the central limit theorem, that Ramanujan graphs verify 
the quantitative Shannon--McMillan--Breiman type theorem 
(the statement without $\cN_{\delta\log|G|}$ in Corollary~\ref{qsmb}). 
In case of random walks on transitive graphs, one has 
$\IE{f_t^{1/2}}=\IE{g_t^{1/2}}$, where $g_t:=\frac{P^t(X^t,X^0)}{P^{t-1}(X^t,X^1)}$, 
and $g_{t+1}\approx\IE{g_t\,|\, X^1,X^{t+1}}$.
In view of Kaimanovich and Vershik's proof of the Shannon--McMillan--Breiman 
type theorem (Theorem~2.1 in \cite{kv}), it would be interesting to know 
whether or not $g_t\approx\IE{g_t\,|\,X_1,\ldots,X_m}$ for $m=o(\log|G|)$. 

In general, for a finitely generated (infinite) group $\Gamma$ 
and a finitely-supported non-degenerate symmetric probability 
measure $\mu$ on $\Gamma$, the value 
$\ip{\lambda(\mu)(\mu^{t-1})^{1/2},(\mu^t)^{1/2}}$ 
converges to $\ip{\sigma(\mu)1_\Pi,1_\Pi}_{L^2(\Pi,\nu)}$, where $(\Pi,\nu)$ 
is the Poisson boundary and $\sigma$ is the unitary Koopman representation.
This value can be strictly smaller than the spectral radius $\rho$. For example, 
if $G$ is amenable, then $\rho=1$; 
while $\ip{\lambda(\mu)(\mu^{t-1})^{1/2},(\mu^t)^{1/2}}\to1$ if and only if 
the Poisson boundary is trivial. See \cite{kv}.
\end{remark}

\subsection*{Acknowledgments}
The author would like to thank the anonymous referee for valuable comments 
and Remark~\ref{rem:referee}.
A part of this work was done during the research camp in Tsurui village 
in August 2020. 
The author acknowledges the kind hospitality of the people of Tsurui village. 
The camp was supported by Operator Algebra Supporters' Fund and 
JSPS KAKENHI 17K05277. 

\end{document}